\documentclass[12pt,reqno]{amsart}
\usepackage{amssymb}

\newtheorem{theorem}{Theorem}[section]
\newtheorem{lemma}[theorem]{Lemma}
\newtheorem{proposition}[theorem]{Proposition}
\newtheorem{corollary}[theorem]{Corollary}

\numberwithin{equation}{section}

\begin{document}

\baselineskip=15pt

\title[Logarithmic connections on bundles on a Riemann surface]{Logarithmic connections on 
principal bundles over a Riemann surface}

\author[I. Biswas]{Indranil Biswas}

\address{School of Mathematics, Tata Institute of Fundamental
Research, Homi Bhabha Road, Mumbai 400005, India}

\email{indranil@math.tifr.res.in}

\author[A. Dan]{Ananyo Dan}

\address{BCAM -- Basque Centre for Applied Mathematics, Alameda de Mazarredo 14,
48009 Bilbao, Spain} 

\email{adan@bcamath.org}

\author[A. Paul]{Arjun Paul} 

\address{School of Mathematics, Tata Institute of Fundamental 
Research, Homi Bhabha Road, Mumbai 400005, India} 

\email{apmath90@math.tifr.res.in}

\author[A. Saha]{Arideep Saha}

\address{School of Mathematics, Tata Institute of Fundamental
Research, Homi Bhabha Road, Mumbai 400005, India}

\email{arideep@math.tifr.res.in}

\subjclass[2010]{53B15, 14H60, 32A27.}

\keywords{Logarithmic connection, residue, automorphism, maximal tori.}

\begin{abstract}
Let $E_G$ be a holomorphic principal $G$--bundle on a compact connected Riemann 
surface $X$, where $G$ is a connected reductive complex affine algebraic group. Fix a 
finite subset $D\, \subset\, X$, and for each $x\,\in\, D$ fix $w_x\, \in\, 
\text{ad}(E_G)_x$. Let $T$ be a maximal torus in the group of all holomorphic 
automorphisms of $E_G$. We give a necessary and sufficient condition for the existence 
of a $T$--invariant logarithmic connection on $E_G$ singular over $D$ such that the 
residue over each $x\, \in\, D$ is $w_x$. We also give a necessary and sufficient 
condition for the existence of a logarithmic connection on $E_G$ singular over $D$ 
such that the residue over each $x\, \in\, D$ is $w_x$, under the assumption that 
each $w_x$ is $T$--rigid.
\end{abstract}

\maketitle

\section{Introduction}

Let $X$ be a compact connected Riemann surface. Given a holomorphic vector bundle $E$ on $X$, a 
theorem of Weil and Atiyah says that $E$ admits a holomorphic connection if and only if the degree 
of every indecomposable component of $E$ is zero (see \cite{We}, \cite{At}). Now let $G$ be a 
connected complex affine algebraic group and $E_G$ a holomorphic principal $G$--bundle on $X$. Then 
$E_G$ admits a holomorphic connection if and only if for every holomorphic reduction of structure 
group $E_H\, \subset\, E_G$, where $H$ is a Levi factor of some parabolic subgroup of $G$, and for 
every holomorphic character $\chi$ of $H$, the degree of the associated line bundle
\begin{equation}\label{ese1}
E_H(\chi)\,=\, E_H\times^\chi {\mathbb C}\,\longrightarrow\, X
\end{equation}
is zero \cite{AB}. Our aim here is to investigate the logarithmic 
connections on $E_G$ with fixed residues, where $(G,\, E_G)$ is as above. More precisely, fix a 
finite subset $D\, \subset\, X$ and also fix $$ w_x\, \in\, \text{ad}(E_G)_x $$ for each $x\, \in\, 
D$, where $\text{ad}(E_G)$ is the adjoint vector bundle for $E_G$. We investigate the existence of 
logarithmic connections on $E_G$ singular over $D$ such that residue is $w_x$ for every $x\,\in\, 
D$.

Let $\text{Aut}(E_G)$ denote the group of all holomorphic automorphisms of $E_G$; it is a complex
affine algebraic group. Fix a maximal torus
$$
T\, \subset\, \text{Aut}(E_G)\, .
$$
This choice produces a Levi factor $H$ of a parabolic subgroup of $G$ as well as a holomorphic
reduction of structure group $E_H\, \subset\, E_G$ to $H$ \cite{BBN}. This pair $(H,\, E_H)$ is
determined by $T$ uniquely up to a holomorphic automorphism of $E_G$.

The group $\text{Aut}(E_G)$ acts on the vector bundle $\text{ad}(E_G)$. An element of
$\text{ad}(E_G)$ will be called $T$--rigid if it is fixed by the action of $T$; some
examples are given in Section \ref{ex}.

We prove the following (see Theorem \ref{thm2}):

\begin{theorem}\label{intt1}
The following two are equivalent:
\begin{enumerate}
\item There is a $T$--invariant logarithmic connection
on $E_G$ singular over $D$ with residue $w_x$ at every $x\, \in\, D$.

\item The element $w_x$ is $T$--rigid for each $x\,\in\, D$, and
$$
{\rm degree}(E_H(\chi))+\sum_{x\in D} d\chi(w_x)\,=\, 0
$$
for every holomorphic character $\chi$ of $H$, where $E_H(\chi)$ is the line bundle
in \eqref{ese1}, and $d\chi\,:\, {\rm Lie}(H)\,\longrightarrow\,
{\mathbb C}$ is the homomorphism of Lie algebras corresponding to $\chi$.
\end{enumerate}
\end{theorem}

The Lie algebra $\mathbb C$ being abelian the homomorphism $d\chi$ factors through the
conjugacy classes in $\text{Lie}(H)$, so it can be evaluated on the elements of
$\text{ad}(E_H)$.

We also the prove the following (see Theorem \ref{thm1}):

\begin{theorem}\label{intt2}
Assume that each $w_x$ is $T$--rigid. Then there is a logarithmic connection on $E_G$ singular
over $D$, with residue $w_x$ at every $x\, \in\, D$, if and only if
$$
{\rm degree}(E_H(\chi))+\sum_{x\in D} d\chi(w_x)\,=\, 0\, ,
$$
where $E_H(\chi)$ and $d\chi(w_x)$ are as in Theorem \ref{intt1}.
\end{theorem}

\section{Logarithmic connections and residue}

\subsection{Preliminaries}\label{se2.1}

Let $G$ be a connected reductive affine algebraic group defined over $\mathbb C$. A 
Zariski closed connected subgroup $P\, \subset\, G$ is called a parabolic subgroup 
if $G/P$ is a projective variety \cite[11.2]{Bo}, \cite{Hu}. The unipotent radical 
of a parabolic subgroup $P\, \subset\, G$ will be denoted by $R_u(P)$. The quotient 
group $P/R_u(P)$ is called the \textit{Levi quotient} of $P$. A \textit{Levi factor} 
of $P$ is a Zariski closed connected subgroup $L\, \subset\, P$ such that the 
composition $L\, \hookrightarrow\, P\, \longrightarrow\, P/R_u(P)$ is an isomorphism 
\cite[p.~184]{Hu}. We note that $P$ admits Levi factors, and any two Levi factors of 
$P$ are conjugate by an element of $R_u(P)$ \cite[\S~30.2, p.~185, Theorem]{Hu}.

The multiplicative group ${\mathbb C}\setminus\{0\}$ will be denoted by
${\mathbb G}_m$. A torus is a product of copies of ${\mathbb G}_m$. Any
two maximal tori in a complex algebraic group are conjugate \cite[p.~158,
Proposition 11.23(ii)]{Bo}.

By a homomorphism between algebraic groups or by a character we will always mean a
holomorphic homomorphism or a holomorphic character.

\subsection{Logarithmic connections}\label{se2.2}

Let $X$ be a compact connected Riemann surface. Fix a finite subset
$$
D\,:=\, \{x_1,\, \cdots,\, x_n\}\, \subset\, X\, .
$$
The reduced effective divisor $x_1+\ldots + x_n$ will also be denoted by $D$.

Let
\begin{equation}\label{g1}
p\, :\, E_H\,\longrightarrow\, X
\end{equation}
be a holomorphic principal $H$--bundle on $X$, where $H$ is a
connected affine algebraic group defined over $\mathbb C$. The Lie algebra
of $H$ will be denoted by $\mathfrak h$. Let
\begin{equation}\label{dp}
\mathrm{d}p\, :\, \mathrm{T}E_H\, \longrightarrow\, p^*\mathrm{T}X
\end{equation}
be the differential of the map $p$ in \eqref{g1}, where
$\mathrm{T}E_H$ and $\mathrm{T}X$ are the holomorphic tangent bundles
of $E_H$ and $X$ respectively; note that $\mathrm{d}p$ is
surjective. The action of $H$ on $E_H$ produces an action of $H$ on $\mathrm{T}
E_H$. This action on $\mathrm{T}E_H$ clearly preserves the subbundle $\text{kernel}(\mathrm{d}p)$.
Define
$$
\text{ad}(E_H)\, :=\, \text{kernel}(\mathrm{d}p)/H\, \longrightarrow\, X\, ,
$$
which is a holomorphic vector bundle on $X$; it is called the
adjoint vector bundle for $E_H$. We note that $\text{ad}(E_H)$ is identified
with the vector bundle $E_H\times^H {\mathfrak h}\,\longrightarrow\, X$
associated to $E_H$ for the adjoint action of
$H$ on its Lie algebra $\mathfrak h$. So the fibers of $\text{ad}(E_H)$ are
Lie algebras isomorphic to $\mathfrak h$. Define the Atiyah bundle for $E_H$
$$
\text{At}(E_H)\,:=\, (\mathrm{T}E_H)/H\, \longrightarrow\, X\, .
$$
The action of $H$ on $\mathrm{T}E_H$ produces an action of $H$
on the direct image $p_*\mathrm{T}E_H$. We note that
$$
\text{At}(E_H)\,=\, (p_*\mathrm{T}E_H)^H \, \subset\, p_*\mathrm{T}E_H
$$
(see \cite{At}). Taking quotient by $H$, the homomorphism $\mathrm{d}p$
in \eqref{dp} produces a short exact sequence
\begin{equation}\label{at1}
0\, \longrightarrow\, \text{ad}(E_H)\, \longrightarrow\,\text{At}(E_H)\,
\stackrel{\mathrm{d}'p}{\longrightarrow}\, \mathrm{T}X\, \longrightarrow\, 0\, ,
\end{equation}
where $\mathrm{d}'p$ is constructed from $\mathrm{d}p$; this is known as
the Atiyah exact sequence for $E_H$.

The subsheaf $\mathrm{T}X\otimes {\mathcal O}_X(-D)$ of $\mathrm{T}X$ will be denoted
by $\mathrm{T}X(-D)$. Now define
$$
\text{At}(E_H,\,D)\,:=\, (\mathrm{d}'p)^{-1}(\mathrm{T}X(-D))
\,\subset\, \text{At}(E_H)\, ,
$$
where $\mathrm{d}'p$ is the projection in \eqref{at1}. So from \eqref{at1} we have
the exact sequence of vector bundles on $X$
\begin{equation}\label{e3}
0\, \longrightarrow\, \text{ad}(E_H)\, \stackrel{i_0}{\longrightarrow}\,
\text{At}(E_H,\,D)\, \stackrel{\sigma}{\longrightarrow}\, \mathrm{T}X(-D)\,
\longrightarrow\, 0\, ,
\end{equation}
where $\sigma$ is the restriction of $\mathrm{d}'p$; this will be called
the {\it logarithmic Atiyah exact sequence} for $E_H$.

A logarithmic connection on $E_H$ singular over $D$ is a holomorphic homomorphism
\begin{equation}\label{th}
\theta\, :\, \mathrm{T}X(-D)\, \longrightarrow\, \text{At}(E_H,\, D)
\end{equation}
such that $\sigma\circ\theta\,=\, \text{Id}_{\mathrm{T}X(-D)}$, where
$\sigma$ is the homomorphism in \eqref{e3}. Note that giving such a homomorphism
$\theta$ is equivalent to giving a homomorphism
$\varpi\, :\, \text{At}(E_H,\, D)\, \longrightarrow\, \text{ad}(E_H)$ such that
$\varpi\circ i_0\,=\, \text{Id}_{\text{ad}(E_H)}$, where $i_0$ is the homomorphism
in \eqref{e3}.

\subsection{Residue of a logarithmic connection}\label{se2.3}

Given a vector bundle $W$ on $X$, the fiber of $W$ over any point $x\, \in\, X$
will be denoted by $W_x$. For any ${\mathcal O}_X$--linear homomorphism
$f\, :\, W\,\longrightarrow\, V$ of holomorphic vector bundles, its restriction
$W_x\, \longrightarrow\, V_x$ will be denoted by $f(x)$.

From \eqref{at1} and \eqref{e3} we have the commutative diagram of homomorphisms
\begin{equation}\label{ef}
\begin{matrix}
0 & \longrightarrow & \text{ad}(E_H) & \stackrel{i_0}{\longrightarrow} & \text{At}(E_H,\,D)
& \stackrel{\sigma}{\longrightarrow} &\mathrm{T}X(-D) & \longrightarrow & 0\\
&& \Vert && ~\Big\downarrow j && ~\Big\downarrow \iota\\
0 & \longrightarrow & \text{ad}(E_H) & \stackrel{i}{\longrightarrow} & \text{At}(E_H)
& \stackrel{\mathrm{d}'p}{\longrightarrow} &\mathrm{T}X & \longrightarrow & 0\\
\end{matrix}
\end{equation}
on $X$. So for any point $x\, \in\, X$, we have
$$
\mathrm{d}'p (x)\circ j(x)\,=\, \iota (x)\circ \sigma (x)\,:\, \text{At}(E_H,\,D)_x\, \longrightarrow\,
(\mathrm{T}X)_x\,=\, \mathrm{T}_xX\, .
$$
Note that $\iota (x)\,=\, 0$ if $x\, \in\, D$, therefore in that case
$\mathrm{d}'p (x)\circ j(x)\,=\, 0$. Consequently, for every $x\, \in\, D$ there is a homomorphism
\begin{equation}\label{R}
R_x\, :\, \text{At}(E_H,\,D)_x\, \longrightarrow\, \text{ad}(E_H)_x
\end{equation}
uniquely defined by the identity $i(x)\circ R_x(v)\,=\, j(x)(v)$ for all $v\, \in\,
\text{At}(E_H,\,D)_x$. Note that
$$
R_x\circ i_0(x)\, =\, \text{Id}_{\text{ad}(E_H)_x}\, ,
$$
where $i_0$ is the homomorphism in \eqref{ef}. Therefore, from \eqref{e3} we have
\begin{equation}\label{ef2}
\text{At}(E_H,\,D)_x\,=\, \text{ad}(E_H)_x\oplus \text{kernel}(R_x)\,=\,
\text{ad}(E_H)_x\oplus \mathrm{T}X(-D)_x\, ;
\end{equation}
note that the composition $\text{kernel}(R_x)\, \hookrightarrow\,
\text{At}(E_H,\,D)_x \, \stackrel{\sigma(x)}{\longrightarrow}\,
\mathrm{T}X(-D)_x$ is an isomorphism.

For any $x\,\in\, D$, the fiber $\mathrm{T}X(-D)_x$ is identified with $\mathbb C$ 
using the Poincar\'e adjunction formula \cite[p.~146]{GH}. Indeed, for any 
holomorphic coordinate $z$ around $x$ with $z(x)\,=\, 0$, the image of 
$z\frac{\partial}{\partial z}$ in $\mathrm{T}X(-D)_x$ is independent of the choice 
of the coordinate function $z$; the above mentioned identification between 
$\mathrm{T}X(-D)_x$ and $\mathbb C$ sends this independent image to $1\,\in\, 
\mathbb C$. Therefore, from \eqref{ef2} we have
\begin{equation}\label{ef3}
\text{At}(E_H,\,D)_x\,=\,\text{ad}(E_H)_x\oplus{\mathbb C}
\end{equation}
for all $x\, \in\, D$.

For a logarithmic connection $\theta\, :\, \mathrm{T}X(-D)\, \longrightarrow\, \text{At}(E_H,\, D)$
as in \eqref{th}, and any $x\, \in\, D$, define
\begin{equation}\label{res}
\text{Res}(\theta,\, x)\,:=\, R_x(\theta(1))\, \in\, \text{ad}(E_H)_x\, ,
\end{equation}
where $R_x$ is the homomorphism in \eqref{R}; in the above definition $1$ is 
considered as an element of $\mathrm{T}X(-D)_x$ using the identification of 
$\mathbb C$ with $\mathrm{T}X(-D)_x$ mentioned earlier.

The element $\text{Res}(\theta,\, x)$ in \eqref{res} is called the \textit{residue},
at $x$, of the logarithmic connection $\theta$.

\subsection{Extension of structure group}\label{se2.4}

Let $M$ be a complex affine algebraic group and
$$
\rho\, :\, H\, \longrightarrow\, M
$$
a homomorphism. As before, $E_H$ is a holomorphic principal $H$--bundle on $X$. Let
$$
E_M\,:=\,E_H\times^\rho M\, \longrightarrow\, X
$$
be the holomorphic principal $M$--bundle obtained by extending the structure group of $E_H$ using
$\rho$. So $E_M$ is the quotient of $E_H\times M$ obtained by identifying
$(y,\, m)$ and $(yh^{-1},\,\rho(h)m)$, where $y$, $m$ and $h$ run over
$E_H$, $M$ and $H$ respectively. Therefore, we have a morphism
$$
\widehat{\rho}\, :\, E_H\, \longrightarrow\, E_M\, , \ \ y\, \longmapsto\, 
\widetilde{(y,\, e_M)}\, ,
$$
where $\widetilde{(y,\, e_M)}$ is the equivalence class of $(y,\, e_M)$ with $e_M$ being the 
identity element of $M$. The homomorphism of Lie algebras $d\rho\,:\,{\mathfrak h}\, 
\longrightarrow\, {\mathfrak m}\, :=\, \text{Lie}(M)$ associated to $\rho$ produces a homomorphism of vector bundles
\begin{equation}\label{alpha}
\alpha\, :\, \text{ad}(E_H)\, \longrightarrow\, \text{ad}(E_M)\, .
\end{equation}
The maps $\widehat{\rho}$ and $d\rho$ together produce
a homomorphism of vector bundles
$$
\widetilde{A}\, :\, \text{At}(E_H)\, \longrightarrow\, \text{At}(E_M)\, .
$$
This map $\widetilde{A}$ produces a homomorphism
\begin{equation}\label{A}
A\, :\, \text{At}(E_H,\, D)\, \longrightarrow\, \text{At}(E_M,\, D)\, ,
\end{equation}
which fits in the following commutative diagram of homomorphisms
\begin{equation}\label{comA}
\begin{matrix}
0 & \longrightarrow & \text{ad}(E_H) & \stackrel{i_0}{\longrightarrow} & \text{At}(E_H,
\,D) & \stackrel{\sigma}{\longrightarrow} &\mathrm{T}X(-D) & \longrightarrow & 0\\
&& ~\Big\downarrow\alpha && ~\Big\downarrow A && \Vert \\
0 & \longrightarrow & \text{ad}(E_M) & \longrightarrow & \text{At}(E_M,\, D)
& \longrightarrow &\mathrm{T}X(-D) & \longrightarrow & 0\\
\end{matrix}
\end{equation}
where the top exact sequence is the one in \eqref{e3} and the bottom one is
the corresponding sequence for $E_M$.

If $\theta\, :\, \mathrm{T}X(-D)\, \longrightarrow\, \text{At}(E_H,\, D)$ is a
logarithmic connection on $E_H$ as in \eqref{th}, then
\begin{equation}\label{ic}
A\circ\theta\, :\, \mathrm{T}X(-D)\, \longrightarrow\, \text{At}(E_M,\, D)
\end{equation}
is a logarithmic connection on $E_M$ singular over $D$. From the definition of
residue in \eqref{res} it follows immediately that 
\begin{equation}\label{res2}
\alpha(\text{Res}(\theta,\, x))\,=\, \text{Res}(A\circ\theta,\, x)
\end{equation}
for all $x\, \in\, D$. This proves the following:

\begin{lemma}\label{lem2}
With the above notation, if $E_H$ admits a logarithmic connection 
$\theta$ singular over $D$ with residue $w_x \in {\rm ad}(E_H)_x$ 
at each $x \,\in\, D$, then $E_M$ admits a logarithmic connection 
$\theta'\, =\, A\circ\theta$ singular over $D$ with 
residue $\alpha(w_x)$ at each $x \,\in \,D$.
\end{lemma}

\section{Connections with given residues}

\subsection{Formulation of residue condition}

Fix a holomorphic principal $H$--bundle $E_H$ on $X$, and fix elements
$$
w_x\, \in\, \text{ad}(E_H)_x
$$
for all $x\, \in\, D$. Consider the decomposition of $\text{At}(E_H,\,D)_x$ in
\eqref{ef3}. For any $x\, \in\, D$, let
$$
{\ell}_x\, :=\, {\mathbb C}\cdot (w_x, \, 1)\, \subset\, 
\text{ad}(E_H)_x\oplus{\mathbb C}\,=\, \text{At}(E_H,\,D)_x
$$
be the line in the fiber $\text{At}(E_H,\,D)_x$. Let
$$
{\mathcal A}\, \subset\, \text{At}(E_H,\,D)
$$
be the subsheaf that fits in the short exact sequence
\begin{equation}\label{cA}
0\, \longrightarrow\, {\mathcal A}\, \longrightarrow\, \text{At}(E_H,\,D)
\, \longrightarrow\, \bigoplus_{x\in D} \text{At}(E_H,\,D)_x/{\ell}_x
\, \longrightarrow\, 0\, .
\end{equation}
Note that the composition
$$
\text{ad}(E_H)_x\, \stackrel{i_0(x)}{\longrightarrow}\, \text{At}(E_H,\,D)_x
\, {\longrightarrow}\, \text{At}(E_H,\,D)_x/\ell_x
$$
is injective, hence it is an isomorphism, where $i_0$ is the homomorphism in \eqref{comA}; this
composition will be denoted by $\phi_x$. Therefore, from \eqref{e3} and \eqref{cA} we have a
commutative diagram
\begin{equation}\label{cdc}
\begin{matrix}
&& 0 && 0 && 0\\
&& \Big\downarrow &&\Big\downarrow && \Big\downarrow\\
0&\longrightarrow & \text{ad}(E_H)\otimes{\mathcal O}_X(-D)
&\longrightarrow & {\mathcal A}&\stackrel{\sigma_1}{\longrightarrow} & \mathrm{T}X(-D)
&\longrightarrow & 0\\
&& \Big\downarrow &&~ \Big\downarrow\nu && ~\,~\Big\downarrow{\rm id}\\
0&\longrightarrow & \text{ad}(E_H) &\stackrel{i_0}{\longrightarrow} &
\text{At}(E_H,\,D) &\stackrel{\sigma}{\longrightarrow} & \mathrm{T}X(-D)
&\longrightarrow & 0\\
&& \Big\downarrow &&\Big\downarrow && \Big\downarrow\\
0&\longrightarrow & \bigoplus\limits_{x\in D} \text{ad}(E_H)_x
&\stackrel{\bigoplus\limits_{x\in D} \phi_x}{\longrightarrow}
& \bigoplus\limits_{x\in D} \text{At}(E_H,\,D)_x/{\ell}_x &\longrightarrow & 0
&\longrightarrow & 0\\
&& \Big\downarrow &&\Big\downarrow && \Big\downarrow\\
&& 0 && 0 && 0
\end{matrix}
\end{equation}
where all the rows and columns are exact; the restriction of $\sigma$ to the
subsheaf $\mathcal A$ is denoted by $\sigma_1$.

\begin{lemma}\label{lem0}
Consider the space of all logarithmic connections on $E_H$ singular over $D$
such that the residue over every $x\, \in\, D$ is $w_x$. It is in bijection with the
space of all holomorphic splittings of the short exact sequence of vector bundles
$$
0\,\longrightarrow\,{\rm ad}(E_H)\otimes{\mathcal O}_X(-D)\,
\longrightarrow \, {\mathcal A}\,\stackrel{\sigma_1}{\longrightarrow} \, \mathrm{T}X(-D)
\,\longrightarrow \,0
$$
on $X$ in \eqref{cdc}.
\end{lemma}

\begin{proof}
Let $\theta\, :\, \mathrm{T}X(-D)\, \longrightarrow\, \text{At}(E_H,\, D)$
be a logarithmic connection on $E_H$ singular over $D$
such that the residue over every $x\, \in\, D$ is $w_x$. From the definition of residue
and the construction of $\mathcal A$ it follows that
$$
\theta(\mathrm{T}X(-D))\, \subset\, {\mathcal A}\, \subset\,\text{At}(E_H,\, D)\, .
$$
Therefore, $\theta$ defines a holomorphic homomorphism
$$
\theta'\, :\, \mathrm{T}X(-D)\, \longrightarrow\, {\mathcal A}\, .
$$
Evidently, we have $\sigma_1\circ\theta'\,=\, \text{Id}_{\mathrm{T}X(-D)}$. So
$\theta'$ is a holomorphic splitting of the exact sequence in the lemma.

To prove the converse, let
$$
\theta_1\, :\, \mathrm{T}X(-D)\, \longrightarrow\, {\mathcal A}
$$
be a holomorphic homomorphism such that $\sigma_1\circ\theta_1\,=\,
\text{Id}_{\mathrm{T}X(-D)}$. Consider the composition
$$
\nu\circ\theta_1\, :\, \mathrm{T}X(-D)\, \longrightarrow\, \text{At}(E_H,\, D)\, ,
$$
where $\nu$ is the homomorphism in \eqref{cdc}. This defines a logarithmic connection
on $E_H$ singular over $D$, because $\sigma\circ\nu\circ\theta_1\,=\,
\sigma_1\circ\theta_1\,=\, \text{Id}_{\mathrm{T}X(-D)}$ by the
commutativity of \eqref{cdc}. From \eqref{cdc} it follows
immediately that $\nu\circ\theta_1(\mathrm{T}X(-D)_x)\, =\, \ell_x\, \subset\,
\text{At}(E_H,\,D)_x$ for every $x\, \in\, D$. Now from the definition of residue
it follows that the residue of the connection $\nu\circ\theta_1$ at any
$x\, \in\, D$ is $w_x$.
\end{proof}

\subsection{Extension class}

The short exact sequence in Lemma \ref{lem0} determines a cohomology class 
\begin{equation}\label{ext-cls}
 \beta \,\in\, H^1(X,\, \text{Hom}({\rm T}X(-D),\,{\rm ad}(E_H) \otimes \mathcal{O}_X(-D)))\, 
 =\, H^1(X,\, {\rm ad}(E_H) \otimes K_X)\,,
\end{equation}
where $K_X \,=\, ({\rm T}X)^*$ is the holomorphic cotangent bundle of $X$. 
Therefore, $E_H$ admits a logarithmic connection singular over 
$D$ with residue $w_x \in {\rm ad}(E_H)_x$ at each $x \in D$ if and 
only if the cohomology class $\beta$ in \eqref{ext-cls} vanishes. 

Let $\rho : H \longrightarrow M$ be a homomorphism of affine algebraic 
groups. Let $E_M := E_H \times^{\rho} M$ be the principal $M$--bundle over $X$ 
obtained by extending the structure group of $E_H$ to $M$ by $\rho$. Consider the
homomorphism $\alpha$ in \eqref{alpha}. It produces a homomorphism
$$
\overline{\rho}\,:\, H^1(X,\, {\rm ad}(E_H) \otimes K_X) \longrightarrow H^1(X,\, {\rm ad}(E_M) \otimes K_X)\,.
$$
{}From \eqref{comA} we have a commutative diagram
$$
\begin{matrix}
0 & \longrightarrow & \text{ad}(E_H)\otimes{\mathcal O}_X(-D) & \longrightarrow & {\mathcal A}
& \stackrel{\sigma_1}{\longrightarrow} &\mathrm{T}X(-D) & \longrightarrow & 0\\
&& \Big\downarrow && \Big\downarrow && \Vert \\
0 & \longrightarrow & \text{ad}(E_M)\otimes{\mathcal O}_X(-D) & \longrightarrow & {\mathcal A}(E_M)
& \longrightarrow &\mathrm{T}X(-D) & \longrightarrow & 0\\
\end{matrix}
$$
where the top exact sequence is the one in Lemma \ref{lem0} and the bottom one is the same sequence 
for $E_M$. From this diagram it follows that the cohomology class in $H^1(X,\, {\rm ad}(E_M) \otimes K_X)$ for 
the short exact sequence in Lemma \ref{lem0} for $E_M$ coincides with $\overline{\rho}(\beta)$, 
where $\beta$ is the cohomology class in \eqref{ext-cls}.

\begin{corollary}\label{cor1}
Assume that $E_H$ admits a holomorphic reduction of structure
group $E_J\, \subset\, E_H$ to a complex algebraic subgroup $J \,\subset\, H$.
The cohomology class $\beta$ in \eqref{ext-cls} is contained in the image of the
natural homomorphism $H^1(X,\, {\rm ad}(E_J) \otimes K_X)
\, \hookrightarrow\, H^1(X,\, {\rm ad}(E_H) \otimes K_X)$
\end{corollary}

\subsection{A necessary condition for logarithmic connections with given 
residue}\label{se3}

Let $\theta$ be a logarithmic connection on $E_H$ singular over $D$.
Take any character $$\chi\, :\, H\, \longrightarrow\, {\mathbb G}_m\, .$$
The group $H$ acts on $\mathbb C$; the action of $h\, \in\, H$ sends any
$c\, \in\, \mathbb C$ to $\chi(h)\cdot c$. Let
$$
E_H(\chi)\, :=\, E_H\times^\chi {\mathbb C}\, \longrightarrow\, X
$$
be the holomorphic line bundle over $X$ associated to $E_H$ for this action of $H$ on 
$\mathbb C$. Since ${\mathbb G}_m$ is abelian, the adjoint vector bundle for $E_H(\chi)$ is 
the trivial holomorphic line bundle ${\mathcal O}_X$ over $X$. The above logarithmic 
connection $\theta$ induces a logarithmic connection on $E_H(\chi)$ (see 
\eqref{ic}); this induced logarithmic connection on $E_H(\chi)$ will be denoted by 
$\theta^\chi$.

For any $x\,\in\, D$, let
$$
\text{Res}(\theta^\chi,\, x)\, \in\, \mathbb C
$$
be the residue of $\theta^\chi$. The residue $\text{Res}(\theta,\, x)$ defines a conjugacy
class in the Lie algebra $\mathfrak h$, because any fiber of $\text{ad}(E_H)$ is
identified with $\mathfrak h$ uniquely up to conjugation. From
\eqref{res2} it follows immediately that
$\text{Res}(\theta^\chi,\, x)$ coincides with with $d\chi(\text{Res}(\theta,\, x))$, where
$d\chi\, :\, {\mathfrak h}\,\longrightarrow\, \mathbb C$ is the homomorphism of Lie algebras
associated to $\chi$; note that since the Lie algebra $\mathbb C$ is abelian, the homomorphism
$d\chi$ factors through the conjugacy classes in $\mathfrak h$.

As $\theta^\chi$ is a logarithmic connection on the line bundle $E_H(\chi)$ with
residue $d\chi(\text{Res}(\theta,\, x))$ at each $x\, \in\, D$, using a computation in
\cite{Oh} it follows that
$$
\text{degree}(E_H(\chi))+\sum_{x\in D} d\chi(\text{Res}(\theta,\, x))\,=\, 0
$$
(see \cite[Lemma 2.3]{BDP}).

Therefore, we have the following:

\begin{lemma}\label{lem1}
Let $E_H$ be a holomorphic principal $H$--bundle on $X$. Fix
$$
w_x\, \in\, {\rm ad}(E_H)_x
$$
for every $x\, \in\, D$. If there is a logarithmic connection on $E_H$ singular over $D$ with
residue $w_x$ at every $x\, \in\, D$, then
$$
{\rm degree}(E_H(\chi))+\sum_{x\in D} d\chi(w_x)\,=\, 0
$$
for every character $\chi$ of $H$, where $E_H(\chi)$ is the associated holomorphic
line bundle, and $d\chi$ is the homomorphism of Lie algebras corresponding to $\chi$.
\end{lemma}

Let $\rho\, :\, H\, 
\longrightarrow\, M$ be an injective homomorphism to a connected complex 
algebraic group $M$ and $E_M\, :=\, 
E_H\times^\rho M$ the holomorphic principal $M$--bundle on $X$ obtained by extending 
the structure group of $E_H$ using $\rho$. As before, the Lie algebras of $H$ and $M$
will be denoted by $\mathfrak h$ and $\mathfrak m$ respectively. Using the injective
homomorphism of Lie algebras
\begin{equation}\label{df}
d\rho\, :\, {\mathfrak h}\,\longrightarrow\, {\mathfrak m}
\end{equation}
associated to $\rho$, we have an injective homomorphism $\alpha$ as in \eqref{alpha}.
For every $x\, \in\, D$, fix an element $w_x\, \in\, \text{ad}(E_H)_x$.

\begin{lemma}\label{le-2}
Assume that $H$ is reductive.
There is a logarithmic connection on $E_H$ singular over $D$ with residue $w_x$
at each $x\, \in\, D$ if there is a logarithmic connection on $E_M$ singular over $D$
with residue $\alpha(x)(w_x)$ at each $x\, \in\, D$, where $\alpha$ is the homomorphism
in \eqref{alpha}.
\end{lemma}

\begin{proof}
The adjoint action of $H$ on $\mathfrak h$ makes it an $H$--module. On the other hand,
the homomorphism $\rho$ composed with the adjoint action of $M$ on $\mathfrak m$ produces
an action of $H$ on $\mathfrak m$. The injective homomorphism $d\rho$ in \eqref{df} is
a homomorphism of $H$--modules. Since $H$ is reductive, there is an $H$--submodule
$V\, \subset\, \mathfrak m$ which is a complement of $d\rho({\mathfrak h})$, meaning
$$
{\mathfrak m}\,=\, d\rho({\mathfrak h})\oplus V\, .
$$
Let $\eta\, :\, {\mathfrak m}\,\longrightarrow\, {\mathfrak h}$ be the projection
constructed from this decomposition of $\mathfrak m$; in particular, we have
$\eta\circ d\rho\,=\, \text{Id}_{\mathfrak h}$.

Since the above $\eta$ is a homomorphism of $H$--modules, it produces a projection
$$
\widehat{\eta}\, :\, \text{ad}(E_M)\, \longrightarrow\, \text{ad}(E_H) 
$$
such that $\widehat{\eta}\circ\alpha\,=\, \text{Id}_{\text{ad}(E_H)}$, where $\alpha$ is
the homomorphism in \eqref{alpha}.

Now if $\theta\, :\, \text{At}(E_M,\, D)\, \longrightarrow\, \text{ad}(E_M)$
is a logarithmic connection on $E_M$ singular over $D$
with residue $\alpha(x)(w_x)$ at each $x\, \in\, D$, consider the composition
$$
\widehat{\eta}\circ\theta\circ A\, :\, \text{At}(E_H,\, D)\, \longrightarrow\,
\text{ad}(E_H)\, ,
$$
where $A$ is constructed in \eqref{A}. Evidently, it is a
logarithmic connection on $E_H$ singular over $D$ with residue $w_x$
at each $x\, \in\, D$.
\end{proof}

\section{$T$--rigid elements of adjoint bundle}\label{se4}

\subsection{Definition}\label{se4.1}

As before, $H$ is a complex affine algebraic group and $p\, :\, E_H\, \longrightarrow\, X$ 
a holomorphic principal $H$--bundle on $X$. An automorphism of $E_H$ is a holomorphic
map $F\, :\, E_H\, \longrightarrow\, E_H$ such that
\begin{itemize}
\item $p\circ F\,=\, p$, and

\item $F(zh)\,=\, F(z)h$ for all $z\, \in\, E_H$ and $h\, \in\, H$.
\end{itemize}
Let $\text{Aut}(E_H)$ be the group of all automorphisms of $E_H$. We will show that
$\text{Aut}(E_H)$ is a complex affine algebraic group.

First consider the case of $H\,=\, \text{GL}(r,{\mathbb C})$. For a holomorphic principal 
$\text{GL}(r,{\mathbb C})$--bundle $E_{\rm GL}$ on $X$, let $E\,:=\, E_{\rm GL}\times^{\text{GL}( 
r,{\mathbb C})} {\mathbb C}^r$ be the holomorphic vector bundle of rank $r$ on $X$ associated to 
$E_{\rm GL}$ for the standard action of $\text{GL}(r,{\mathbb C})$ on ${\mathbb C}^r$. Then 
$\text{Aut}(E_{\rm GL})$ is identified with the group of all holomorphic automorphisms $\text{Aut} 
(E)$ of the vector bundle $E$ over the identity map of $X$. Note that $\text{Aut}(E)$ is the 
Zariski open subset of the complex affine space $H^0(X,\, \text{End}(E))$ consisting of all global 
endomorphisms $f$ of $E$ such that $\det (f(x_0))\, \not=\, 0$ for a fixed point $x_0\, \in \, X$; 
since $x\, \longmapsto\, \det (f(x))$ is a holomorphic function on $X$, it is in fact a constant 
function. Therefore, $\text{Aut}(E_{\rm GL})$ is an affine algebraic variety over $\mathbb C$.

For a general $H$, fix an algebraic embedding $\rho\, :\, H\, \hookrightarrow\, 
\text{GL}(r,{\mathbb C})$ for some $r$. For a holomorphic principal $H$--bundle 
$E_H$ on $X$, let $E_{\rm GL}\,:=\, E_H\times^\rho\, \text{GL}(r,{\mathbb C})$ be 
the holomorphic principal $\text{GL}(r,{\mathbb C})$--bundle on $X$ obtained by extending 
the structure group of $E_H$ using $\rho$. The injective homomorphism $\rho$ 
produces an injective homomorphism
$$
\rho'\, :\, \text{Aut}(E_H)\, \longrightarrow\, \text{Aut}(E_{\rm GL})\, .
$$
The image of $\rho'$ is Zariski closed in the
algebraic group $\text{Aut}(E_{\rm GL})$. Hence $\rho'$ produces the structure of a 
complex affine algebraic group on $\text{Aut}(E_H)$. This structure 
of a complex algebraic group is independent 
of the choices of $r$, $\rho$. Therefore, $\text{Aut}(E_H)$ is an affine algebraic 
group. Note that $\text{Aut}(E_H)$ need not be connected, although the automorphism
group of a holomorphic vector bundle is always connected (as it is a Zariski open subset
of a complex affine space).

The Lie algebra of $\text{Aut}(E_H)$ is $H^0(X, \, \text{ad}(E_H))$. The group $\text{Aut}(E_H)$ 
acts on any fiber bundle associated to $E_H$. In particular, $\text{Aut}(E_H)$ acts on the adjoint 
vector bundle $\text{ad}(E_H)$. This action evidently preserves the Lie algebra structure on the 
fibers of $\text{ad}(E_H)$.

Let $\text{Aut}(E_H)^0\, \subset\, \text{Aut}(E_H)$ be the connected component
containing the identity element. Fix a maximal torus
$$
T\, \subset\, \text{Aut}(E_H)^0\, .
$$
An element $w\, \in\, \text{ad}(E_H)_x$, where $x\, \in\, X$, will be called
$T$--\textit{rigid} if the action of $T$ on $\text{ad}(E_H)_x$ fixes $w$.

Consider the adjoint action of $H$ on itself. Let
\begin{equation}\label{gs}
\text{Ad}(E_H)\,:=\, E_H\times^H H\, \longrightarrow\, X
\end{equation}
be the associated holomorphic fiber bundle. Since this adjoint action preserves the group structure 
of $H$, the fibers of $\text{Ad}(E_H)$ are complex algebraic groups isomorphic to $H$. More 
precisely, each fiber of $\text{Ad}(E_H)$ is identified with $H$ uniquely up to an inner 
automorphism of $H$. The corresponding Lie algebra bundle on $X$ is $\text{ad}(E_H)$.

The group $\text{Aut}(E_H)$ is the space of all holomorphic sections of $\text{Ad}(E_H)$.
For any $x\, \in\, X$, the action of $\text{Aut}(E_H)$ on the fiber $\text{ad}(E_H)_x$
coincides with the one obtained via the composition
$$
\text{Aut}(E_H)\, \stackrel{{\rm ev}_x}{\longrightarrow}\, \text{Ad}(E_H)_x
\, \stackrel{\rm ad}{\longrightarrow}\, {\rm Aut}(\text{ad}(E_H)_x)\, ,
$$
where ${\rm ev}_x$ is the evaluation map that sends a section of $\text{Ad}(E_H)$ to its
evaluation at $x$, and ${\rm ad}$ is the adjoint action of the group $\text{Ad}(E_H)_x$ on its
Lie algebra $\text{ad}(E_H)_x$. 

Therefore, an element $w\, \in\, \text{ad}(E_H)_x$ is
$T$--rigid if and only if the adjoint action of ${\rm ev}_x(T)\, \subset\,
\text{Ad}(E_H)_x$ on $\text{ad}(E_H)_x$ fixes $w$.

\subsection{Examples}\label{ex} The center of $H$ will be denoted by $Z_H$. Let $E_H$ be
a holomorphic principal $H$--bundle on $X$. Since $Z_H$ commutes
with $H$, for any $t\, \in\, Z_H$, the map $E_H\,\longrightarrow\, E_H$, $z\, \longmapsto\,
zt$ is $H$--equivariant. Therefore, we have $Z_H\, \subset\, \text{Aut}(E_H)$. The principal
$H$--bundle $E_H$ is called \textit{simple} if $Z_H\,=\,\text{Aut}(E_H)$.
Note that if $E_H$ is simple then every element of $\text{ad}(E_H)$ is $T$--rigid, where $T$
is any maximal torus in $\text{Aut}(E_H)^0$.

Let $H$ be connected reductive, and let $E_H$ be stable. Then $Z_H$ is a finite index
subgroup of $\text{Aut}(E_H)$. Therefore, any maximal torus of $\text{Aut}(E_H)^0$ is
contained in $Z_H$. This implies that every element of $\text{ad}(E_H)$ is $T$--rigid, where $T$
is any maximal torus in $\text{Aut}(E_H)^0$.

Take $E_H$ to be the trivial holomorphic principal $H$--bundle $X\times H$. Then the 
left--translation action of $H$ identifies $H$ with $\text{Aut}(E_H)$. Also, $\text{ad}(E_H)$ is 
the trivial vector $X\times\mathfrak h$, where $\mathfrak h$ is the Lie algebra of $H$. Let $T$ be 
a maximal torus of $H\,=\, \text{Aut}(E_H)$. Then an element $v\,\in\, {\mathfrak h}\,=\, 
\text{ad}(E_H)_x$ is $T$--rigid if and only if $v\,\in\, \text{Lie}(T)$.

\section{A criterion for logarithmic connections with given residue}\label{se5}

\subsection{Logarithmic connections with $T$--rigid residue}

As in Section \ref{se2.1}, $G$ is a connected reductive affine algebraic group defined over 
$\mathbb C$. Let $E_G$ be a holomorphic principal $G$--bundle over $X$. Fix a maximal
torus
$$
T\, \subset\, \text{Aut}(E_G)^0\, ,
$$
where $\text{Aut}(E_G)^0$ as before is the connected component containing the identity 
element of the group of automorphisms of $E_G$. 

We now recall some results from \cite{BBN}, \cite{BP}.

As in \eqref{gs}, define the adjoint bundle $\text{Ad}(E_G)\, =\, E_G\times^G G$. 
For any point $y\, \in\, X$, consider the evaluation homomorphism
$$
\varphi_y\, :\, T\, \longrightarrow\, \text{Ad}(E_G)_y\, ,\ \ s\, \longmapsto\, s(y)\, .
$$
Then $\varphi_y$ is injective and its image is a torus in $G$ \cite[p.~230,
Section~3]{BBN}. Since
$G$ is identified with $\text{Ad}(E_G)_y$ uniquely up to an inner automorphism, the image
$\varphi_y(T)$ determines a conjugacy class of tori in $G$; this conjugacy class is
independent of the choice of $y$ \cite[p.~230, Section~3]{BBN},
\cite[p.~63, Theorem~4.1]{BP}. Fix a torus
\begin{equation}\label{tg}
T_G\, \subset\, G
\end{equation}
in this conjugacy class of tori. The centralizer
\begin{equation}\label{cz}
H\, :=\, C_G(T_G)\, \subset\, G
\end{equation}
of $T_G$ in $G$ is a Levi factor of a parabolic subgroup of $G$ \cite[p.~230, 
Section~3]{BBN}, \cite[p.~63, Theorem~4.1]{BP}. The principal $G$--bundle 
$E_G$ admits a holomorphic reduction of structure group
\begin{equation}\label{re}
E_H\, \subset\, E_G
\end{equation}
to the above subgroup $H$ \cite[p.~230, Theorem~3.2]{BBN}, \cite[p.~63,
Theorem~4.1]{BP}. Since $T_G$ is in the center of $H$, the action of $T_G$ on
$E_H$ commutes with the action of $H$, so $T_G\,\subset\, \text{Aut}^0(E_H)$
(this was noted in Section \ref{ex}). The image of $T_G$ in $\text{Aut}^0(E_H)$
coincides with $T$. This reduction $E_H$ is minimal in the sense that there is no
Levi factor $L$ of some parabolic subgroup of $G$ such that
\begin{itemize}
\item $L\, \subsetneq\, H$, and

\item $E_G$ admits a holomorphic reduction of structure group to $L$.
\end{itemize}
(See \cite[p.~230, Theorem~3.2]{BBN}.)

The above reduction $E_H$ is unique in the following sense. Let $L$ be a Levi factor of a 
parabolic subgroup of $G$ and $E_L\, \subset\, E_G$ a holomorphic reduction of structure group to 
$L$ satisfying the condition that $E_G$ does not admit any holomorphic reduction of structure 
group to a Levi factor $L'$ of some parabolic subgroup of $G$ such that $L'\,\subsetneq\, L$.
Then there is an automorphism $\varphi\, \in\, \text{Aut}(E_G)^0$ such 
that $E_L\,=\, \varphi(E_H)$ \cite[p.~63, Theorem~4.1]{BP}. In particular, the 
subgroup $L\, \subset\, G$ is conjugate to $H$.

The Lie algebras of $G$ and $H$ will be denoted by $\mathfrak g$ and $\mathfrak h$ respectively. 
The inclusion of $\mathfrak h$ in $\mathfrak g$ and the reduction in \eqref{re} together produce 
an inclusion $\text{ad}(E_H)\, \hookrightarrow\, \text{ad}(E_G)$.
This subbundle $\text{ad}(E_H)$ of $\text{ad}(E_G)$ coincides with the
invariant subbundle $\text{ad}(E_G)^T$
for the action of $T$ on $\text{ad}(E_G)$ \cite[p.~230, Theorem~3.2]{BBN}, 
\cite[p.~61, Proposition~3.3]{BP} (this action is explained in Section \ref{se4.1}),
in other words,
\begin{equation}\label{re2}
\text{ad}(E_H) \,=\, \text{ad}(E_G)^T\, \subset\, \text{ad}(E_G)\, .
\end{equation}

For every $x\, \in\, D$ fix a $T$--rigid element
\begin{equation}\label{wx}
w_x\, \in\, \text{ad}(E_G)_x
\end{equation}
(see Section \ref{se4.1}). Since each $w_x$ is $T$--rigid, from \eqref{re2} we
conclude that
\begin{equation}\label{re3}
w_x\, \in\, \text{ad}(E_H)_x\ \ \ \forall\ x\,\in\, D\, .
\end{equation}
So $w_x$ determines a conjugacy class in $\mathfrak h$. For any character
$\chi$ of $H$, the corresponding homomorphism of Lie algebras
$d\chi\, :\, {\mathfrak h}\, \longrightarrow\, \mathbb C$ factors through the
conjugacy classes in $\mathfrak h$, because $\mathbb C$ is abelian. Therefore, we have
$d\chi(w_x)\, \in\, \mathbb C$.

\begin{theorem}\label{thm1}
There is a logarithmic connection on $E_G$ singular over $D$, and with
$T$--rigid residue $w_x$ at every $x\, \in\, D$ (see \eqref{wx}), if and only if
\begin{equation}\label{si}
{\rm degree}(E_H(\chi))+\sum_{x\in D} d\chi(w_x)\,=\, 0
\end{equation}
for every character $\chi$ of $H$, where $E_H(\chi)$ is the holomorphic
line bundle on $X$ associated to $E_H$ for $\chi$, and $d\chi$ is the homomorphism of
Lie algebras corresponding to $\chi$.
\end{theorem}

\begin{proof}
Assume that there is a logarithmic connection on $E_G$ singular over $D$ such that the residue at 
each $x\, \in\, D$ is $w_x$. Since the group $H$ is reductive, from Lemma \ref{le-2} it follows 
that $E_H$ admits a logarithmic connection singular over $D$ such that the residue at each 
$x\,\in\, D$ is $w_x$ (see \eqref{re3}). Now from Lemma \ref{lem1} we know that \eqref{si}
holds for every character $\chi$ of $G$.

To prove the converse, assume that \eqref{si} holds for every character $\chi$ of $H$. We will show 
that $E_G$ admits a logarithmic connection singular over $D$ such that the residue at each $x\, 
\in\, D$ is $w_x$.

Since $E_G$ is the extension of structure group of $E_H$ using the inclusion of 
$H$ in $G$ (see \eqref{re}), a logarithmic connection on $E_H$ induces a 
logarithmic connection on $E_G$ (this is explained in Section \ref{se2.4}). 
Therefore, in view of \eqref{res2} and \eqref{re3}, the
following proposition completes the proof of the theorem.

\begin{proposition}\label{prop1}
There is a logarithmic connection on $E_H$ singular over $D$ such that
the residue over any $x\, \in\, D$ is $w_x\, \in\, {\rm ad}(E_H)_x$.
\end{proposition}

\begin{proof}
The connected component of the center of $H$ containing the identity element coincides
with $T_G$ in \eqref{tg}. Define the quotient groups
$$
S\, :=\, H/T_G\, ,\ \ \ Z\,:=\, H/[H,\, H]\, .
$$
So $S$ is semisimple, and $Z$ is a torus. The projections of $H$ to $S$ and $Z$ 
will be denoted by $p_S$ and $p_Z$ respectively. Let $E_S$ (respectively, $E_Z$)
be the principal $S$--bundle (respectively, $Z$--bundle) on $X$ obtained by extending
the structure group of $E_H$ using $p_S$ (respectively, $p_Z$). Consider the homomorphism
\begin{equation}\label{vp}
\varphi\, :\, H\, \longrightarrow\, S\times Z\,, \ \ \ h\,\longmapsto\,
(p_S(h),\, p_Z(h))\, .
\end{equation}
It is surjective with finite kernel, hence it induces an isomorphism of Lie algebras.
Let $E_{S\times Z}$ be the principal $S\times Z$--bundle on $X$ obtained by extending
the structure group of $E_H$ using $\varphi$. Note that $E_{S\times Z}\, \cong\, E_S\times_X E_Z$.
Since $\varphi$ induces an isomorphism of Lie algebras, we have
\begin{equation}\label{eq-5.2.1}
 \text{ad}(E_H)\,=\, \text{ad}(E_{S\times Z})\,=\, \text{ad}(E_S)\oplus \text{ad}(E_Z)
\end{equation}
and
$$
\text{At}(E_H)\,=\, \text{At}(E_{S\times Z})\, , \ \ \
{\mathcal A}\,=\, \mathcal{A}_{E_{S \times Z}}\, ,
$$
where $\mathcal{A}_{E_{S \times Z}}$ is constructed as in \eqref{cA} for $(E_{S \times Z},\, 
\{w_x\}_{x\in D})$ (see \eqref{eq-5.2.1}). Consequently, $E_H$ admits a logarithmic connection 
singular over $D$, with residue $w_x$ for all $x\, \in\, D$, if and only if $E_{S\times Z}$ admits 
a logarithmic connection singular over $D$ with residue $w_x$ for all $x\, \in\, D$.

For $x\,\in\, D$, let
$$
w_x\,=\, w^s_x\oplus w^z_x\, , \ \ \ w^s_x\, \in\, \text{ad}(E_S)_x\, ,\ \
w^z_x\, \in\, \text{ad}(E_Z)_x
$$
be the decomposition given by \eqref{eq-5.2.1}. Consider the short exact sequences
\begin{equation}\label{ex-1}
 0 \longrightarrow {\rm ad}(E_S) \otimes \mathcal{O}_X(-D) \longrightarrow \mathcal{A}_{E_S} \stackrel{\sigma_{1,S}}{\longrightarrow} \mathrm{T}X(-D) \longrightarrow 0
\end{equation}
and 
\begin{equation}\label{ex-2}
 0 \longrightarrow {\rm ad}(E_Z) \otimes \mathcal{O}_X(-D) \longrightarrow \mathcal{A}_{E_Z} \stackrel{\sigma_{1,Z}}{\longrightarrow} \mathrm{T}X(-D) \longrightarrow 0\,,
\end{equation}
as in Lemma \ref{lem0} for the data $(E_S,\, \{w_x^s\}_{x \in D})$ and $(E_Z,\, \{w_x^z\}_{x\in D})$ respectively.
Let
$$q_S \,:\, \mathcal{A}_{E_S} \oplus \mathcal{A}_{E_Z} \,\longrightarrow \,\mathcal{A}_{E_S}\, ,\ \ 
q_Z \,:\, \mathcal{A}_{E_S} \oplus \mathcal{A}_{E_Z} \,\longrightarrow\, \mathcal{A}_{E_Z}$$
be the projections. Note that
$$
\mathcal{A}_{E_S} \oplus \mathcal{A}_{E_Z}\, \supset\,
\text{kernel}(\sigma_{1,S} \circ q_S - \sigma_{1,Z}\circ q_Z)\, =\,
\mathcal{A}_{E_S} \times_{TX(-D)} \mathcal{A}_{E_Z}\, =\, \mathcal{A}_{E_{S \times Z}}\, .$$

Therefore, giving a holomorphic splitting of the exact sequence
$$
0 \,\longrightarrow\, {\rm ad}(E_{S \times Z}) \otimes \mathcal{O}_X(-D)
\,\longrightarrow\, \mathcal{A}_{E_{S \times Z}} \,\longrightarrow\, TX(-D)\,\longrightarrow\, 0
$$
(see Lemma \ref{lem0} for it) is equivalent to giving holomorphic splittings of
both \eqref{ex-1} and \eqref{ex-2}. Consequently, $E_{S\times Z}$ admits a logarithmic connection
singular over $D$ with residue $w_x$ over every $x\, \in\, D$ if and only if both
$E_{S}$ and $E_Z$ admit logarithmic connections singular over $D$ such that
their residues over any $x\, \in\, D$ are $w^s_x$ and $w^z_x$ respectively.

Consider the homomorphism of character groups $\text{Hom}(Z,\, {\mathbb 
G}_m)\,\longrightarrow\, \text{Hom}(H,\, {\mathbb G}_m)$ given by the projection 
$p_Z$. It is an isomorphism because being semisimple $[H,\, H]$ does not admit any nontrivial 
character. Since $E_Z \,=\, E_H \times^{p_Z} Z$, for any character $\chi \,\in\, 
\text{Hom}(Z,\, \mathbb{G}_m)$, the holomorphic line bundle $E_Z(\chi) \,=\, E_Z \times^{\chi} 
\mathbb{C}$ is identified with the holomorphic line bundle $E_H(\chi\circ p_Z)$. 

A holomorphic line bundle $L$ on $X$ admits a logarithmic connection singular over
$D$ with residue $\lambda_x\, \in\, \mathbb C$ for every $x\, \in\, D$ if and only if
$$
\text{degree}(L)+\sum_{x\in D} \lambda_x\,=\, 0
$$
(see \cite[Lemma 2.3]{BDP}). Since $Z = H/[H,\, H]\,=\, (\mathbb{G}_m)^d$ for some $d$, it follows 
that $E_Z$ admits a logarithmic connection singular over $D$ with residue $w_x^z$ at each $x 
\,\in\, D$ if and only if for each $1\, \leq\, i\, \leq\, d$, the line bundle $E_Z(\pi_i)$ admits a 
logarithmic connection singular over $D$ with residue ${\rm d}\pi_i(w_x^z)$ at each $x \in D$, 
where $\pi_i\,:\, Z\,=\,(\mathbb{G}_m)^d\,\longrightarrow\, {\mathbb G}_m$ is the projection
to the $i$-th factor. From this and the given condition in \eqref{si} we conclude that $E_Z$
admits logarithmic connection singular over $D$ with residue $w^z_x$ for all $x\,\in\, D$.

To complete the proof of the proposition we need to show that $E_S$ admits logarithmic connection 
singular over $D$ such that the residues over each $x\, \in\, D$ is $w^s_x$. We will show that 
\eqref{ex-1} splits holomorphically.

Let
\begin{equation}\label{beta}
\beta\, \in\, H^1(X,\, \text{Hom}(\mathrm{T}X(-D),\,
{\rm ad}(E_S)\otimes{\mathcal O}_X(-D)))\,=\, H^1(X,\, {\rm ad}(E_S)\otimes K_X)
\end{equation}
be the extension class for \eqref{ex-1} as in \eqref{ext-cls}.
The exact sequence in \eqref{ex-1} splits holomorphically if and only if
\begin{equation}\label{beta2}
\beta\,=\, 0\, .
\end{equation}

The Lie algebra of $S$ will be denoted by $\mathfrak s$.
Consider $\mathfrak s$ as a $S$--module using
the adjoint action of $S$ on $\mathfrak{s}$.
Since $S$ is semisimple, the Killing form
$$
\kappa \,:\, \mathfrak{s} \times \mathfrak{s} \,\longrightarrow\, \mathbb{C}\, , \ \ 
~ (v,\, w) \,\longmapsto\, \text{trace}({\rm ad}_v \circ {\rm ad}_w)\,,
$$
is nondegenerate, where ${\rm ad}_u(u') \,:=\, [u,\, u']$. Therefore, the Killing form
induces an isomorphism ${\mathfrak s}\, \stackrel{\sim}{\longrightarrow}\, {\mathfrak s}^*$ of
$S$--modules. This isomorphism produces a holomorphic isomorphism of ${\rm ad}(E_S)$ with the
dual vector bundle ${\rm ad}(E_S)^*$. Now Serre duality gives
$$
H^1(X,\, {\rm ad}(E_S)\otimes K_X)\,=\, H^0(X,\, {\rm ad}(E_S)^*)^*\,=\,
H^0(X,\, {\rm ad}(E_S))^*\, .
$$
Let
$$
\beta'\,\in\, H^0(X,\, {\rm ad}(E_S))^*
$$
be the element corresponding to $\beta$ (defined in \eqref{beta}) by the
above isomorphism. Then 
\begin{equation}\label{eq-1}
 \beta'(\gamma) \,=\, \int_X \kappa(\widehat{\beta},\, \gamma)\, ,\ \ \forall\ 
\gamma \,\in \, H^0(X,\, {\rm ad}(E_S))\,,
\end{equation}
where $\widehat{\beta}$ is an ${\rm ad}(E_S)$--valued $(1,\,1)$--form on $X$ 
which represents the cohomology class $\beta$ using the Dolbeault isomorphism.

As before, $\text{Aut}(E_H)^0\, \subset\, \text{Aut}(E_H)$ (respectively, $\text{Aut}(E_G)^0\, 
\subset\, \text{Aut}(E_G)$) is the connected component containing the identity element, and $T 
\,\subset\, \text{Aut}(E_G)^0$ is the fixed maximal torus. Since $T$ is abelian, from \eqref{re2} it 
follows immediately that $$T\, \subset\, \text{Aut}(E_H)^0\, \subset\, \text{Aut}(E_G)^0\, .$$ 
Therefore, the maximal torus $T\, \subset\, \text{Aut}(E_G)^0$ (see \eqref{tg}) is also a maximal 
torus of $\text{Aut}(E_H)^0$. Since $T_G$ is the connected component, containing the identity 
element, of the center of $H$, and $T$ is the image of $T_G$ in $\text{Aut}(E_H)^0$,
it now follows that the maximal torus of $\text{Aut}(E_S)^0$ is trivial. 
Hence every holomorphic section of $\text{ad}(E_S)$ is nilpotent.

Take any nonzero element $\gamma\, \in\, H^0(X,\, {\rm ad}(E_S))$. Following the proof 
of \cite[Proposition 3.9]{AB}, using $\gamma$ we construct a 
holomorphic reduction of the structure group of $E_S$ to a parabolic 
subgroup of $S$ as follows. For each $x \,\in\, X$, since $\gamma(x)
\,\in \, {\rm ad}(E_S)_x$ is nilpotent, there is a parabolic Lie subalgebra 
$\mathfrak{p}_x \,\subset \,{\rm ad}(E_S)_x$ canonically associated to 
$\gamma(x)$ \cite[p.~340, Lemma~3.7]{AB}. Exponentiating $\mathfrak{p}_x$ 
we get a proper parabolic subgroup $P_x \,\subset\, {\rm Ad}(E_S)_x$ associated 
to $\gamma(x)$. Since there are only finitely many conjugacy classes of nilpotent 
elements of $\mathfrak{s}$, and the algebraic subvariety of $\mathfrak{s}$ defined 
by the nilpotent elements has a natural filtration defined using the adjoint action of $S$, 
there is a finite subset $C \,\subset\, X$ such that the conjugacy classes of $P_x$,
$x \,\in\, X \setminus C$, coincide. Fix a parabolic subgroup $P\, \subset\, S$
in this conjugacy class.

For any $x\, \in\, X \setminus C$, consider the projection map
$$
\xi_x \,:\, (E_S)_x \times S \,\longrightarrow \,{\rm Ad}(E_S)_x \,, \ \ (z,\, s)
\,\longmapsto\, \widetilde{(z,\, s)}\, ,
$$
where $\widetilde{(z,\, s)}$ is the equivalence class of $(z,\, s)$. Define 
$$(E_P)_x \,:= \,\{z \,\in\, (E_S)_x \,\mid\, \xi_x(z,\, g)\,\in\, P_x\,,\ \forall \
g \,\in\, P\}\,.$$
For the natural action of $S$ on $(E_S)_x$, the action of
$P \,\subset\, S$ preserves $(E_P)_x$. Since $P$ is a parabolic subgroup of 
$S$, its normalizer $N_S(P)$ is $P$ itself \cite[p.~143, Corollary B]{Hu}. So the 
action of $P$ on $(E_P)_x$ is transitive (and also free, since the $G$--action 
on $(E_S)_x$ is free). Therefore, we have a holomorphic reduction of structure group
$E_P \,\subset \,E_S$ to $P \,\subset\, S$ over $X \setminus C$. This holomorphic
reduction defines a holomorphic section $\eta \,: \,X \setminus C 
\,\longrightarrow\, E_S/P$ which is meromorphic over $X$.
Since $S/P$ is a projective variety, the above section $\eta$ extends 
holomorphically to a section $\widetilde{\eta} \,: \,X \,\longrightarrow\, E_S/P$. 
This defines a holomorphic reduction of structure group $E_P \,\subset\, E_S$
to $P$.

From Corollary \ref{cor1} it follows that the cohomology class 
$\beta$ in \eqref{beta} lies in the image of the natural homomorphism 
$H^1(X,\, {\rm ad}(E_P) \otimes K_X)\, \longrightarrow\,
H^1(X,\, {\rm ad}(E_S) \otimes K_X)$. Therefore, $\beta$ is represented by an
${\rm ad}(E_P)$--valued $(1,\,1)$--form $\widehat{\beta}$ on $X$.

Now for all $x \,\in \,X \setminus C$, the element
$\gamma(x)\, \in\, {\rm ad}(E_S)_x$ lies in
the Lie algebra ${\mathfrak u}_x$ of the unipotent radical of ${\rm ad}(E_P)_x$
\cite[p.~186, Corollary A]{Hu}, and $\widehat{\beta}(x) \,\in \,\mathfrak{p}_x$. 
Therefore, $\kappa(\widehat{\beta}(x),\, \gamma(x)) \,= \,0$, since
$\mathfrak{u}_x$ is the orthogonal complement
$({\rm ad}(E_P)_x)^{\perp}$ with respect to the Killing form on
${\rm ad}(E_S)_x$. Hence from \eqref{eq-1} we have
$\beta'(\gamma)\,=\, 0$. This proves \eqref{beta2}, and 
completes the proof of the proposition.
\end{proof}

As noted before, Proposition \ref{prop1} completes the proof of Theorem
\ref{thm1}.
\end{proof}

\subsection{$T$--invariant logarithmic connections with given residue}

The automorphism group ${\rm Aut}(E_G)$ has a natural action on the space of all 
logarithmic connections on $E_G$ singular over $D$. Given a maximal torus $T\, 
\subset\, {\rm Aut}(E_G)^0$, by a $T$--invariant logarithmic connection we mean a 
logarithmic connection on $E_G$ singular over $D$ which is fixed by the action of 
$T$.

\begin{theorem}\label{thm2}
Let $E_G$ be a holomorphic principal $G$--bundle on $X$, where $G$ is reductive.
Fix $w_x\, \in\, {\rm ad}(E_G)_x$ for each $x\, \in\, D$. Fix a maximal torus
$T\, \subset\, {\rm Aut}(E_G)^0$. The following two are equivalent:
\begin{enumerate}
\item There is a $T$--invariant logarithmic connection
on $E_G$ singular over $D$ with residue $w_x$ at every $x\, \in\, D$.

\item The element $w_x$ is $T$--rigid for each $x\,\in\, D$, and
\eqref{si} holds for every character $\chi$ of $H$.
\end{enumerate}
\end{theorem}

\begin{proof}
Let $\theta$ be a $T$--invariant logarithmic connection
on $E_G$ singular over $D$ with residue $w_x$ at every $x\, \in\, D$. Since
$\theta$ is $T$--invariant, its residues are also $T$--invariant. Hence
$w_x$ is $T$--rigid for each $x\,\in\, D$. From Theorem \ref{thm1} we know
that \eqref{si} holds for every character $\chi$ of $H$.

Now assume that the second statement in the theorem holds. From Theorem \ref{thm1} we know that 
there is a logarithmic connection on $E_G$ singular over $D$ with residue $w_x$ at every $x\, 
\in\, D$.

As noted in Section \ref{ex}, for a holomorphic principal $M$--bundle $E_M$ on $X$, the center 
$Z_M$ of $M$ is contained in the automorphism group $\text{Aut}(E_M)$. It is straight-forward to 
check that the action of $Z_M\, \subset\, \text{Aut}(E_M)$ on the space of all logarithmic 
connections on $E_M$ is trivial.

Since $T_G$ is contained in the center of $H$ (see \eqref{cz}), and $T$ is the image of $T_G$ in 
$\text{Aut}(E_H)$, every logarithmic connection on the principal $H$--bundle $E_H$ in Proposition 
\ref{prop1} is $T$--invariant. Consequently, from Proposition \ref{prop1} it follows that $E_G$ 
admits a $T$--invariant logarithmic connection singular over $D$ with residue $w_x$ at every $x\, 
\in\, D$.
\end{proof}

\section*{Acknowledgements}

The first author is supported by a J. C. Bose Fellowship.
The second author is supported by ERCEA Consolidator Grant $615655$-NMST and also
by the Basque Government through the BERC $2014-2017$ program and by Spanish
Ministry of Economy and Competitiveness MINECO: BCAM Severo Ochoa
excellence accreditation SEV-$2013-0323$.


\end{document}